\documentclass[11pt]{amsart}

\usepackage[white,ams]{tms}

\usepackage{amssymb}

\newtheorem{thm}{Theorem}[section]

\theoremstyle{remark}
\newtheorem{rmk}{Remark}[section]
\newtheorem{ex}[rmk]{Example}

\title[$\G$-Structures, $H$-Flatness and Integrability]{$\G$-Structures in Dimension Four,\\ $H$-Flatness and Integrability}
\author[W.~Kry\'nski]{Wojciech Kry\'nski}
\author[T.~Mettler]{Thomas Mettler}
\address{
Institute of Mathematics, Polish Academy of Sciences, ul. \'Sniadeckich 8, 00-656 Warszawa, Poland}
\email{krynski@impan.pl}
\address{Institut f\"ur Mathematik, Goethe-Universit\"at Frankfurt, Robert-Mayer-Str.~10, 60325 Frankfurt am Main, Germany}
\email{mettler@math.uni-frankfurt.de}

\renewcommand{\d}{\mathrm{d}}
\newcommand{\R}{\mathbb{R}}
\newcommand{\V}{\mathcal{V}}
\newcommand{\G}{\mathrm{GL}(2)}
\newcommand{\p}{\partial}

\date{November 24, 2016}

\begin{document}

\maketitle

\begin{abstract}
We show that torsion-free four-dimensional $\G$-stru\-ctu\-res are flat up to a coframe transformation with a mapping taking values in a certain subgroup $H\subset\mathrm{SL}(4,\R)$ which is isomorphic to a semidirect product of the three-dimensional continuous Heisenberg group $H_3(\R)$ and the Abelian group $\R$. In addition, we show that the relevant PDE system is integrable in the sense that it admits a dispersionless Lax-pair.   
\end{abstract}

\section{Introduction}

A $\G$-structure on a smooth $4$-manifold $M$ is given by a smoothly varying family of twisted cubic curves, one in each projectivised tangent space of $M$. Equivalently, a $\G$-structure is the same as $G$-structure $\pi\colon B \to M$ on $M$, where $G$ is the image subgroup of the faithful irreducible $4$-dimensional representation of $\mathrm{GL}(2,\R)$ on the space of homogeneous polynomials of degree three with real coefficients in two real variables. In particular, a $\G$-structure is called~\textit{torsion-free} if its associated $G$-structure is torsion-free. Torsion-free $\G$-structures are of particular interest as they provide examples of torsion-free connections with exotic holonomy group $\mathrm{GL}(2,\R)$. However,  the local existence of torsion-free $\G$-structures is highly non-trivial, even when applying the Cartan--K\"ahler machinery which is particularly well-suited for the construction of torsion-free connections with special holonomy. Adapting methods of Hitchin~\cite{MR699802}, Bryant~\cite{MR1141197} gave an elegant twistorial construction of real-analytic torsion-free $\G$-structures in dimensions four, thus providing the first example of an irreducibly-acting holonomy group of a (non-metric) torsion-free connection missing from Berger's list~\cite{MR0079806} of such connections.  

A natural source for $\G$-structures are differential operators. Recall that the principal symbol $\sigma(\mathrm{D})$ of a $k$-th order linear differential operator $\mathrm{D} \colon C^{\infty}(M,\R^n) \to C^{\infty}(M,\R^m)$ assigns to each point $p \in M$ a homogeneous polynomial of degree $k$ on $T^*_pM$ with values in $\mathrm{Hom}(\R^n,\R^m)$. Therefore, in each projectivised cotangent space $\mathbb{P}(T^*_pM)$ of $M$ we obtain the so-called~\textit{characteristic variety} $\Xi_p$ of $\mathrm{D}$, consisting of those  $[\xi] \in \mathbb{P}(T^*_pM)$ for which the linear mapping $\sigma_{\xi}(\mathrm{D}) \colon \R^n \to \R^m$ fails to be injective. In particular, given a (possibly non-linear) differential operator $\mathrm{D}$ and a smooth $\R^n$-valued function $u$ defined on some open subset $U\subset M$ and which satisfies $\mathrm{D}(u)=0$, we may ask that the linearisation $\mathrm{L}_u(\mathrm{D})$ of $\mathrm{D}$ around $u$ has characteristic varieties all of which are the dual variety of the twisted cubic curve. Consequently, one obtains a $\G$-structure on the domain of definition of each solution $u$ of the PDE $\mathrm{D}(u)=0$ for an appropriate class of differential operators. Various examples of such operators have recently been given by Ferapontov--Kruglikov~\cite{arXiv:1607.01966}. In particular, they show that locally all torsion-free $\G$-structures arise in this fashion for some~\textit{second order} operator $\mathrm{D}$ which furthermore has the property that the PDE $\mathrm{D}(u)=0$ admits a dispersionless Lax pair. We  also refer the reader to~\cite{MR3263506} for an application of similar ideas to the case of three-dimensional Einstein--Weyl structures.

Here we show that if a $4$-manifold $M$ carries a torsion-free $\G$-structure $\pi \colon B \to M$, then for every point $p \in M$ there exists a $p$-neighbourhood $U_p$, local coordinates $x \colon U_p \to \R^4$ and a mapping $h \colon U_p \to H$ into a certain $4$-dimensional subgroup $H\subset \mathrm{SL}(4,\R)$, so that the coframing $\eta=h\, \d x$ is a local section of $\pi \colon B \to M$. Moreover, the mapping $h$ satisfies a~\textit{first order} quasi-linear PDE system which admits a dispersionless Lax-pair. Note that our result shows that $4$-dimensional torsion-free $\G$-structures are $H$-\textit{flat}, that is, flat up to a coframe transformation with a mapping taking values in $H$.

Along the way (see Theorem~\ref{ppn:mainpde}), we derive a first order quasi-linear PDE describing general $H$-flat torsion-free $G$-structures which may be of independent interest. 

\subsection*{Acknowledgments} The authors would like to thank Maciej Dunajski and Evgeny Ferapontov for helpful conversations and correspondence.  

\section{$G$-structures and $H$-flatness}

In this section we collect some elementary facts about $G$-structures, introduce the notion of $H$-flatness and derive the first order PDE system describing $H$-flat torsion-free $G$-structures. Throughout the article all manifolds and maps are assumed to be smooth, that is $C^{\infty}$. 

\subsection{The coframe bundle and $G$-structures}

Let $M$ be an $n$-manifold and $V$ a real $n$-dimensional vector space. A coframe at $p \in M$ is a linear isomorphism $f \colon T_pM \to V$. The set $F_pM$ of coframes at $p \in M$ is the fibre of the principal right $\mathrm{GL}(V)$ coframe bundle $\upsilon \colon FM \to M$ where the right action $R_a \colon FM \to FM$ is defined by the rule $R_a(f)=a^{-1}\circ f$ for all $a \in \mathrm{GL}(V)$ and $f \in FM$. Of course, we may identify $V\simeq \R^n$, but it is often advantageous to allow $V$ to be an abstract vector space in which case we say $FM$ is \textit{modelled} on $V$. The coframe bundle carries a tautological $V$-valued $1$-form defined by $\omega_f=f\circ \upsilon^{\prime}$ so that we have the equivariance property $R_a^*\omega=a^{-1}\omega$. A local $\upsilon$-section $\eta \colon U \to FM$ is called a~\textit{coframing} on $U\subset M$ and a choice of a basis of $V$ identifies $\eta$ with $n$ linearly independent $1$-forms on $U$. 

Let $G\subset \mathrm{GL}(V)$ be a closed subgroup. A $G$-structure on $M$ is a reduction $\pi \colon B \to M$ of the coframe bundle with structure group $G$, equivalently a section of the fibre bundle $FM/G \to M$. Note that $V$ is equipped with a canonical coframing $\eta_0$ defined by the exterior derivative of the identity map $\eta_0=\d\, \mathrm{Id}_V$. Consequently, the coframe bundle of $V$ may naturally be identified with $V \times \mathrm{GL}(V)$ and hence the space of $G$-structures on $V$ is in one-to-one correspondence with the maps $V \to \mathrm{GL}(V)/G$. In particular, a map $h \colon V \to \mathrm{GL}(V)$ defines a $G$-structure on $V$ by composing $h$ with the quotient projection $\mathrm{GL}(V) \to \mathrm{GL}(V)/G$. 

\subsection{$H$-flatness}

A $G$-structure $\pi \colon B \to M$ is called~\textit{flat} if in a neighbourhood $U_p$ of every point $p \in M$ there exist local coordinates $x\colon U_p \to V$ so that $\d x \colon U_p \to FM$ takes values in $B$. Suppose $H\subset \mathrm{GL}(V)$ is a closed subgroup. We say a $G$-structure is $H$-\textit{flat} if in a neighbourhood $U_p$ of every point $p \in M$ there exist local coordinates $x \colon U_p \to V$ and a mapping $h\colon U_p \to H$ so that $h\,\d x\colon U_p \to FM$ takes values in $B$. Clearly, every $G$-structure is $\mathrm{GL}(V)$-flat and a $G$-structure is flat in the usual sense if and only if it is $\{e\}$-flat where $\{e\}$ denotes the trivial subgroup of $\mathrm{GL}(V)$. 

\begin{ex}
Every $\mathrm{O}(2)$-structure is $\R^+$-flat where $\R^+$ denotes the group of uniform scaling transformations of $\R^2$ with positive scale factor. This is the existence of local isothermal coordinates for Riemannian metrics in two-dimensions. Likewise, conformally flat Riemannian metrics yield examples of $\mathrm{O}(n)$-structures that are $\R^+$-flat. 
\end{ex}

\subsection{A PDE for $H$-flat torsion-free $G$-structures}

Recall that a $G$-structure $\pi \colon B \to M$ is called~\textit{torsion-free} if there exists a principal $G$-connection $\theta$ on $B$ so that Cartan's first structure equation
\begin{equation}\label{eq:firststruc}
\d \omega=-\theta\wedge\omega
\end{equation}
holds. 
\begin{rmk}
We remark that a weaker notion of torsion-freeness is also in use, see for instance~\cite{MR1334205,MR2003610}. Namely, a $G$-structure $\pi \colon B \to M$ is called torsion-free if there exists a pseudo-connection on $B$ with vanishing torsion, that is, a $1$-form $\theta$ on $B$ with values in the Lie algebra of $G$ so that~\eqref{eq:firststruc} holds.    
\end{rmk}
We may ask when a $G$-structure on $V$ induced by a mapping $h \colon V \to H\subset \mathrm{GL}(V)$ is torsion-free. To this end let $A\subset V^*\otimes V$ be a linear subspace. Denote by
$$
\delta\colon V^*\otimes V^*\otimes V \to \Lambda^2(V^*)\otimes V
$$
the natural skew-symmetrisation map. Recall that the Spencer cohomology group $H^{0,2}(A)$ of $A$ is the quotient
$$
H^{0,2}(A)=\left(\Lambda^2(V^*)\otimes V\right)/\delta(V^*\otimes A).  
$$
Let
$$
\Pi_A \colon \Lambda^2(V^*)\otimes V \to H^{0,2}(A)
$$
denote the quotient projection and let $\mu_H$ denote the Maurer-Cartan form of $H$. Note that $\psi_h=h^*\mu_H$ is a $1$-form on $V$ with values in the Lie algebra $\mathfrak{h}$ of $H$, that is, a smooth map
$$
\psi_h \colon V \to V^*\otimes \mathfrak{h}\subset V^*\otimes \mathfrak{gl}(V)\simeq V^*\otimes V^*\otimes V.
$$
We define $\tau_h=\delta\,\psi_h$, so that $\tau_h$ is a $2$-form on $V$ with values in $V$. Moreover, let $\mathrm{Ad}(h) \colon \mathfrak{gl}(V) \to \mathfrak{gl}(V)$ denote the adjoint action of $h \in H$. We now have:
\begin{thm}\label{ppn:mainpde}
Let $h \colon V \to H$ be a smooth map. Then the $G$-structure defined by $h$ is torsion-free if and only 
\begin{equation}\label{eq:mainpde}
\Pi_{\mathrm{Ad}(h^{-1})\mathfrak{g}}\,\tau_h=0. 
\end{equation}
\end{thm}
\begin{rmk}
In the case where $H=G$ the $H$-structure defined by $h$ is the same as the torsion-free $H$-structure defined by the map $h\equiv \mathrm{Id}_V \colon V \to \mathrm{GL}(V)$, hence~\eqref{eq:mainpde} must be trivially satisfied. This is indeed the case, for any map $h \colon V \to H$ we obtain
$$
\Pi_{\mathrm{Ad}(h^{-1})\mathfrak{h}}\,\tau_h=\Pi_{\mathfrak{h}}\,\tau_h=\Pi_{\mathfrak{h}}\, \delta\,\psi_h=0,
$$
since the adjoint action of $H$ preserves $\mathfrak{h}$. 
\end{rmk}
\begin{proof}[Proof of Proposition~\ref{ppn:mainpde}]
For the proof we fix an identification $V\simeq \R^n$. Let $x=(x^i)$ denote the standard coordinates on $\R^n$. Furthermore let $h \colon \R^n \to H\subset \mathrm{GL}(n,\R)$ be given and let $\pi \colon B_h \to \R^n$ denote the $G$-structure defined by $h$, that is, 
$$
B_h=\left\{(x,a) \in \R^n\times \mathrm{GL}(n,\R) \, : a=h^{-1}(x)g, \;g \in G\, \right\}.
$$ 
Note that we have an $G$-bundle isomorphism
$$
\psi \colon \R^n\times G \to B_h, \quad (x,g)\mapsto (x,h^{-1}(x) g).
$$
The tautological $1$-form $\omega_0$ on $F\R^n\simeq \R^n\times \mathrm{GL}(n,\R)$ satisfies $(\omega_0)_{(x,a)}=a^{-1}\d x$ for all $(x,a) \in \R^n\times \mathrm{GL}(n,\R)$. Continuing to write $\omega_0$ for the pullback to $B_h$ of $\omega_0$, we obtain
$$
\omega_{(x,g)}:=(\psi^*\omega_0)_{(x,g)}=g^{-1}h(x)\d x.
$$ 
Let $\alpha$ be any $1$-form on $\R^n$ with values in $\mathfrak{g}$, the Lie-algebra of $G$. We obtain a principal $G$-connection $\theta=(\theta^i_j)$ on $\R^n\times G$ by defining
$$
\theta=g^{-1}\alpha g+g^{-1} \d g
$$
where $g \colon \R^n \times G \to G\subset \mathrm{GL}(n,\R)$ denotes the projection onto the later factor. Conversely, every principal $G$-connection on the trivial $G$-bundle $\R^n\times G$ arises in this fashion. The $G$-structure $B_h$ is torsion-free if and only if there exists a principal $G$-connection $\theta$ such that
$$
\d \omega+\theta\wedge\omega=0
$$ 
which is equivalent to
$$
0=\d \left(g^{-1}h\d x\right)+\left(g^{-1}\alpha g+g^{-1}\d g\right)\wedge g^{-1}h \d x
$$
or
$$
0=\left(\d g^{-1}+g^{-1}\d g g^{-1}\right)\wedge h\,\d x+g^{-1}\left(\d h\wedge \d x+\alpha\wedge h\,\d x\right)
$$
Using $0=\d\left(g^{-1}g\right)$ we see that the $G$-structure defined by $h$ is torsion-free if and only if there exists a $1$-form $\alpha$ on $V$ with values in $\mathfrak{g}$ such that
$$
0=\d h\wedge \d x +\alpha\wedge h\,\d x.
$$
This is equivalent to
$$
\left(h^{-1} \d h+h^{-1}\alpha h \right)\wedge \d x=0 
$$
or
\begin{equation}\label{eq:torfreecond}
\left(\psi_h+\mathrm{Ad}(h^{-1})\alpha\right)\wedge \d x=0 
\end{equation}
where $\psi_h=h^{-1}\d h$ denotes the $h$-pullback of the Maurer-Cartan form of $H$ and $\mathrm{Ad}(h)v=hvh^{-1}$ the adjoint action of $h\in H$ on $v \in \mathfrak{gl}(n,\R)$. Now~\eqref{eq:torfreecond} is equivalent to
$$
\delta\, \psi_h+\delta\, \mathrm{Ad}(h^{-1})\alpha=0 
$$
Since $\alpha$ takes values in $\mathfrak{g}$ this implies that $\tau_h=\delta\,\psi_h$ lies in the $\delta$-image of $V^*\otimes \mathrm{Ad}(h^{-1})\mathfrak{g}$ which implies
$$
\Pi_{\mathrm{Ad}(h^{-1})\mathfrak{g}}\,\tau_h=0. 
$$
Conversely, suppose $\tau_h$ lies in the $\delta$-image of $V^*\otimes \mathrm{Ad}(h^{-1})\mathfrak{g}$. Then there exists a $1$-form $\beta$ on $V$ with values in $h^{-1}\mathfrak{g}h$ so that
$$
\tau_h=\delta\, \psi_h=\delta\, \beta.
$$
Hence the $\mathfrak{g}$-valued $1$-form $\alpha$ on $V$ defined by $\alpha=-h\beta h^{-1}$ satisfies
$$
\tau_h+\delta\, h^{-1}\alpha h=\delta\, \psi_h + \delta\,\mathrm{Ad}(h^{-1})\alpha=0, 
$$
thus proving the claim. 
\end{proof}

\section{$\G$-structures}

Let $x,y$ denote the standard coordinates on $\R^2$ and let $\R[x,y]$ denote the polynomial ring with real coefficients generated by $x$ and $y$. We let $\mathrm{GL}(2,\R)$ act from the left on $\R[x,y]$ via the usual linear action on $x,y$. We denote by $\mathcal{V}_d$ the subspace consisting of homogeneous polynomials in degree $d\geqslant 0$ and by $G_d\subset \mathrm{GL}(\mathcal{V}_d)$ the image subgroup of the $\mathrm{GL}(2,\R)$ action on $\V_3$. The vector space $\V_3$ carries a two-dimensional cone $\tilde{\mathcal{C}}$ of distinguished polynomials consisting of the perfect cubes, i.e., those that are of the form $(ax+by)^3$ for $ax+by\in \V_1$. The reader may easily check that $G_3$ is characterised as the subgroup of $\mathrm{GL}(\V_3)$ that preserves $\tilde{\mathcal{C}}$. The projectivisation of $\tilde{\mathcal{C}}$ gives an algebraic curve $\mathcal{C}$ of degree $3$ in $\mathbb{P}(\V_3)$ which is linearly equivalent to the~\textit{twisted cubic curve}, i.e., the curve in $\mathbb{RP}^3$ defined by the zero locus of the three homogeneous polynomials
$$
P_0=XZ-Y^2, \quad P_1=YW-Z^2, \quad P_2=XW-YZ, 
$$
where $[X\!:\!Y\!:\!Z\!:\!W]$ are the standard homogeneous coordinates on $\mathbb{RP}^3$. 

Let $M$ be a $4$-manifold and let $\upsilon \colon FM \to M$ denote its coframe bundle modelled on $\V_3$. A $\G$-structure on $M$ is a reduction $\pi \colon B \to M$ of $FM$ with structure group $G_3\simeq \mathrm{GL}(2,\R)$. By definition, a $\G$-structure identifies each tangent space of $M$ with $\V_3$ up to the action by $\mathrm{GL}(2,\R$). Consequently, each projectivised tangent space $\mathbb{P}(T_pM)$ of $M$ carries an algebraic curve $\mathcal{C}_p$ which is linearly equivalent to the twisted cubic curve. Conversely, if $\mathcal{C}\subset \mathbb{P}(TM)$ is a smooth subbundle having the property that each fibre $\mathcal{C}_p$ is linearly equivalent to the twisted cubic curve, then one obtains a unique reduction of the coframe bundle of $M$ whose structure group is $G_3$. 

For what follows it will be convenient to identify $\V_3\simeq \R^4$ by the isomorphism  $\V_3 \to \R^4$ defined on the basis of monomials as 
$$
x^{(3-i)}y^i \mapsto e_{i+1}
$$
where $i=0,1,2,3$ and $e_i$ denotes the standard basis of $\R^4$. Note that, under the identification $T_pM=\V_3$ the cone $\tilde{\mathcal{C}}$ of a $\G$-structure at $p$ can be written as
$$
\tilde{\mathcal{C}}_p=\{s^3e_1+3s^2te_2+3st^2e_3+t^3e_4\ |\ s,t\in\R\}.
$$
We now have:
\begin{thm}\label{thm:mainthm}
All torsion-free $\G$-structures in dimension four are $H$-flat, where $H\subset \mathrm{SL}(4,\R)$ is the subgroup consisting of matrices of the form
\begin{equation}\label{groupH}
\begin{pmatrix} 1 & A & B & D \\ 0 & 1 & A & C \\ 0 & 0 & 1 & A \\ 0 & 0 & 0 & 1\end{pmatrix}
\end{equation}
where $A,B,C,D$ are arbitrary real numbers. 
\end{thm}
\begin{rmk}
We note that the group $H$ is isomorphic to a semidirect product of the continuous three-dimensional Heisenberg group $H_3(\R)$ and the Abelian group $\R$, that is, $H\simeq H_3(\R)\rtimes \R$. Indeed, $H_3(\R)$ has a faithful (necessarily reducible) four-dimensional representation defined by the Lie group homomorphism $\varphi \colon H_3(\R) \to \mathrm{GL}(4,\R)$ 
$$
\begin{pmatrix}
1 & a & c \\ 0 & 1 & b \\ 0 & 0 & 1 
\end{pmatrix}\mapsto \begin{pmatrix} 1 & a & \frac{1}{2}a^2+b & \frac{1}{6}a^3+ab-c\\ 0 & 1 & a & \frac{1}{2}a^2 \\ 0 & 0 & 1 & a \\ 0 & 0 & 0 & 1\end{pmatrix}.
$$
Note that $\varphi$ embeds $H_3(\R)$ as a normal subgroup of the group $H$ and we think of $\R$ as the subgroup of $H$ defined by setting $A=B=D=0$ in~\eqref{groupH}.
\end{rmk}
\begin{rmk}
In fact, the notion of a $\G$-structure makes sense in all dimensions $d\geqslant 3$. However, torsion-free $\G$-structures in dimensions exceeding four are $\{e\}$-flat~\cite{MR1141197}, that is, flat in the usual sense. We refer the reader to~\cite{GN,MR2765729} for a comprehensive study of five-dimensional $\G$-structures (with torsion).   
\end{rmk}
\begin{rmk}
Phrased differently, Theorem~\ref{thm:mainthm} states that locally every tor\-sion-free $\G$-structure in dimension four is obtained from a solution to the quasi-linear first order PDE system~\eqref{eq:mainpde} where $h$ takes values in the aforementioned group $H$.  
\end{rmk}

\begin{proof}[Proof of Theorem~\ref{thm:mainthm}]
We shall prove that for a given torsion-free $\G$-structure one can always choose local coordinates such that the cone  $\tilde{\mathcal{C}}$ has the following form
$$
\tilde{\mathcal{C}}=\{\ s^3V_0+3s^2tV_1+3st^2V_2+t^3V_3|\ s,t\in\R\}
$$
where the framing $(V_0,V_1,V_2,V_3)$ is
\begin{equation}\label{eqVf}
\begin{aligned}
&V_0=\partial_0,\qquad V_1=\partial_1+\alpha\partial_0,\qquad V_2=\partial_2+\alpha\partial_1+\beta\partial_0,\\
&V_3=\partial_3+\alpha\partial_2+\gamma\partial_1+\delta\partial_0,
\end{aligned}
\end{equation}
for some functions $\alpha$, $\beta$, $\gamma$ and $\delta$. Then, the dual coframing is of the form $h\,\mathrm{d}x$, where $h$ takes values in $H$. Indeed, we have
$$
h= \begin{pmatrix}
1 & -\alpha & -\beta+\alpha^2 & -\delta+\alpha(\gamma+\beta)-\alpha^3\\
0 & 1 &  -\alpha &  -\gamma+\alpha^2\\
0 & 0 & 1 & -\alpha \\
0 & 0 & 0 & 1
\end{pmatrix}.
$$
In order to derive the desired form of $\tilde{\mathcal{C}}$ we explore a correspondence between the torsion-free $\G$-structures and classes of contact equivalent fourth order ODEs (compare the proof of \cite[Theorem 1]{DFK} and a similar correspondence in dimension 3). Indeed, it is proved in \cite{MR1141197} that any torsion-free $\G$-structure is defined by a fourth order ODEs of the form
\begin{equation}\label{F}
x^{(4)}=F(y,x,x',x'',x''')
\end{equation}
satisfying the Bryant-W\"unschmann condition which is a system of two non-linear equations for a unknown function $F=F(y,x_0,x_1,x_2,x_3)$ (see also \cite{DT,K1,N}). Here, $(y,x_0,x_1,x_2,x_3)$ are the standard coordinates on the space $J^3(\R,\R)$ of 3-jets of functions $\R\to\R$. The Bryant-W\"unschman condition is invariant with respect to the group of contact transformations of variables $(y,x_0,x_1,x_2,x_3)$. The $\G$-structure corresponding to equation \eqref{F} is defined on the solution space of \eqref{F}, i.e. on the quotient space $J^3(\R,\R)/X_F$, where $X_F=\partial_y+x_1\partial_0+x_2\partial_1+x_3\partial_2+F\partial_3$ is the total derivative. In order to define the structure we first consider the following field of cones on $J^3(\R,\R)$ as in \cite{K0}
$$
\hat{\mathcal{C}}=\{\ s^3\hat V_0+3s^2t\hat V_1+3st^2\hat V_2+t^3\hat V_3\ |\ s,t\in\R\}\mod X_F
$$
where
\begin{align*}
&\hat V_0=\frac{3}{4}\partial_3,\qquad \hat V_1=\frac{1}{2}\partial_2+\frac{3}{8}\partial_3F\partial_3\\
&\hat V_2= \frac{1}{2}\partial_1+\frac{1}{4}\partial_3F\partial_2+ \left(\frac{7}{20}\partial_2F -\frac{3}{20}X_F(\partial_3F) +\frac{9}{40}(\partial_3F)^2\right)\partial_3\\
&\hat V_3=\partial_0+\frac{1}{4}\partial_3F\partial_1+ \left(\partial_2F+\frac{7}{10}K\right)\partial_2\\
&\quad+\bigg(\partial_1F-\frac{3}{10}X_F(K)-X_F(\partial_2F)+\frac{21}{40}K\partial_3F\\
&\quad-\frac{27}{16}X_F(\partial_3F)\partial_3F-\frac{3}{4}\partial_2F\partial_3F+\frac{3}{4}X_F^2(\partial_3F)+\frac{27}{64}(\partial_3F)^3\bigg)\partial_3,
\end{align*}
and $K=-\partial_2F+\frac{3}{2}X(\partial_3F)-\frac{3}{8}(\partial_3F)^2$. To define the cone one looks for $(f,g)$ such that
\begin{equation}\label{eqCone}
\mathrm{ad}_{fX_F}^4(g\p_3)=0\mod X_F,\p_3,\p_2
\end{equation}
where $\mathrm{ad}^i_{X_F}$ stands for the iterated Lie bracket with vector field $X_F$. Then $\hat{\mathcal{C}_p}$ is defined as the set of all $(\mathrm{ad}_{fX_F}^3(g\p_3))(p)$ where $(f,g)$ solve \eqref{eqCone}. The explicit formula for $\hat{\mathcal{C}}$ can be found using \cite[Proposition 4.1]{K0} and \cite[Corollary 5.3]{K0}. The cone $\hat{\mathcal{C}}$ is invariant with respect to the flow of $X_F$ if and only if \eqref{F} satisfies the Bryant-W\"unschmann condition. In this case \eqref{eqCone} takes the form $\mathrm{ad}_{fX_F}^4(g\p_3)=0\mod X_F$ (c.f.~\cite{K1}). Then $\hat{\mathcal{C}}$ can be projected to the quotient space $J^3(\R,\R)/X_F$ and defines a $\G$-structure there via the field of cones $\tilde{\mathcal{C}}=q_*\hat{\mathcal{C}}$, where $q\colon J^3(\R,\R)\to J^3(\R,\R)/X_F$ is the quotient map. Note that  $J^3(\R,\R)/X_F$ can be identified with the hypersurface $\{y=0\}\subset J^3(\R,\R)$. Denoting \begin{align*}
\alpha&=\partial_3F|_{y=0},\\
\beta&=\left(\frac{7}{20}\partial_2F-\frac{3}{20}X(\partial_3F)+\frac{9}{40}(\partial_3F)^2\right)\bigg|_{y=0},\\
\gamma&=\left(\partial_2F+\frac{7}{10}K\right)\bigg|_{y=0},\\
\delta&=\left(\partial_1F-\frac{3}{10}X(K)-X(\partial_2F)+\frac{21}{40}K\partial_3F-\frac{27}{16}X(\partial_3F)\partial_3F\right.\\
&\phantom{=}\left.-\frac{3}{4}\partial_2F\partial_3F+\frac{3}{4}X^2(\partial_3F)+\frac{27}{64}(\partial_3F)^3\right)\bigg|_{y=0}
\end{align*} 
we get that
$$
\tilde{\mathcal{C}}=\{\ s^3V_0+3s^2tV_1+3st^2V_2+t^3V_3\ |\ s,t\in\R\}
$$
where
$$
\begin{aligned}
&V_0=\frac{3}{4}\partial_3,\qquad V_1=\frac{1}{2}\partial_2+\frac{3}{8}\alpha\partial_3,\qquad V_2=\frac{1}{2}\partial_1+\frac{1}{4}\alpha\partial_2+\beta\partial_3,\\
&V_3=\partial_0+\frac{1}{4}\alpha\partial_1+\gamma\partial_2+\delta\partial_3.
\end{aligned}
$$
The following linear change of coordinates
$$
(x_0,x_1,x_2,x_3)\mapsto\left(x_3, 2x_2, 2x_1, \frac{4}{3}x_0\right)
$$
transforms $(V_0,V_1,V_2,V_3)$ to
$$
\begin{aligned}
&V_0=\partial_0,\qquad V_1=\partial_1+\frac{1}{2}\alpha\partial_0,\qquad V_2=\partial_2+\frac{1}{2}\alpha\partial_1+\frac{4}{3}\beta\partial_0,\\
&V_3=\partial_3+\frac{1}{2}\alpha\partial_2+2\gamma\partial_1+\frac{4}{3}\delta\partial_0,
\end{aligned}
$$
which is equivalent to \eqref{eqVf} up to constants.
\end{proof}

\begin{rmk}
It is proved in \cite{arXiv:1607.01966} that locally any torsion-free $\G$-structure admits a coframing of the form $h\,\mathrm{d}x$ with
$$
\begin{aligned}
&h=\\
&{\Small\begin{pmatrix}
a_1a_2a_3& a_0a_2a_3 & a_0a_1a_3 & a_0a_1a_2\\
\frac{1}{3}(a_1a_2b_3+a_1b_2a_3 & \frac{1}{3}(a_0a_2b_3+a_0b_2a_3 & \frac{1}{3}(a_0a_1b_3+a_0b_1a_2 & \frac{1}{3}(a_0a_1b_2+a_0b_1a_3 \\
+b_1a_2a_3) & +b_0a_2a_3) & +b_0a_1a_3) & +b_0a_1a_2) \\
\frac{1}{3}(a_1b_2b_3+b_1a_2b_3 & \frac{1}{3}(a_0b_2b_3+b_0a_2b_3 & \frac{1}{3}(a_0b_1b_3+b_0a_1b_3 & \frac{1}{3}(a_0b_1b_2+b_0a_1b_2 \\
+b_1b_2a_3) & +b_0b_2a_3) & +b_0b_1a_3) & +b_0a_1b_2) \\

b_1b_2b_3& b_0b_2b_3 & b_0b_1b_3 & b_0b_1b_2\end{pmatrix}},
\end{aligned}
$$
where $a_i=\left(\frac{\partial u}{\partial x_i}\right)^{-1}$ and $b_i=\left(\frac{\partial v}{\partial x_i}\right)^{-1}$ for some real-valued functions $u$ and $v$ on $\V_3\simeq \R^4$. It is an interesting problem to find the smallest possible dimension of the group $H$, such that all torsion-free $\G$-structures are $H$-flat.
\end{rmk}

\section{Integrability}
In this section we derive the system \eqref{eq:mainpde} explicitly in terms of the functions $A$, $B$, $C$ and $D$ of Theorem \ref{thm:mainthm}. Moreover, we prove that it possesses a dispersionless Lax pair understood as a pair of commuting vector fields depending on a spectral parameter. Systems of this type, e.g.\,the dispersionless Kadomtsev-Petviashivili equation, often appear as dispersionless limits of integrable PDEs. Other examples include the Pleba\'nski heavenly equation or the Manakov-Santini system describing 3-dimensional Einstein-Weyl geometry. We refer to \cite{MS1,MS2} for general methods of integration of such systems.
\begin{thm}\label{thm:system}
Let $H\subset\mathrm{SL}(4,\R)$ be the subgroup of matrices \eqref{groupH}.
An $H$-flat $\G$-structure defined by a coframing $h\,\mathrm{d}x$, where $h$ takes values in $H$, satisfies \eqref{eq:mainpde}, i.e.\,is torsion-free, if and only if
\begin{equation}\label{intSystem}
\begin{aligned}
&V_2(D)-V_3(B)-AV_2(B)-CV_2(A)+AV_3(A)+A^2V_2(A)=0\\
&2V_1(D)-V_2(C)-2AV_1(B)-V_3(A)+\\
&\qquad\qquad+AV_2(A)+2A^2V_1(A)-2CV_1(A)=0\\
&V_0(D)-2V_1(C)+3V_1(B)-AV_0(B)-2V_2(A)\\
&\qquad\qquad-AV_1(A)-CV_0(A)+A^2V_0(A)=0\\
&V_0(C)-2V_0(B)+V_1(A)+AV_0(A)=0,
\end{aligned}
\end{equation}
where $(V_0,V_1,V_2,V_3)$ is the framing dual to $h\,\mathrm{d}x$ explicitly given by
$$
\begin{aligned}
&V_0=\partial_0,\qquad V_1=\partial_1-A\partial_0,\qquad V_2=\partial_2-A\partial_1-(B-A^2)\partial_0,\\
&V_3=\partial_3-A\partial_2-(C-A^2)\partial_1-(D-(C+B)A+A^3)\partial_0.
\end{aligned}
$$
The system \eqref{intSystem} can be put in the following Lax form
$$
[L_0,L_1]=0,
$$
with
$$
\begin{aligned}
&L_0=\p_3+(-C+2A\lambda-3\lambda^2)\p_1\\
&\qquad+(-D+AC-2A^2\lambda+4A\lambda^2-2\lambda^3)\p_0+N(\lambda)\p_\lambda,\\
&L_1=\p_2+(-A+2\lambda)\p_1+(-B+A^2-2A\lambda+\lambda^2)\p_0+ M(\lambda)\p_\lambda
\end{aligned}
$$
and
$$
\begin{aligned}
N(\lambda)&=\left(\frac{1}{2}A^2A_1-ABA_0+AA_2-AB_1-\frac{1}{2}DA_0-\frac{1}{2}C_2\right.\\
&\qquad\qquad\left.+\frac{1}{2}AC_1+\frac{1}{2}BC_0-\frac{1}{2}CA_1+\frac{1}{2}ACA_0+\frac{1}{2}A_3\right)\\
&\phantom{=}+(3B_1-C_1-AA_1-AC_0+2BA_0-2A_2)\lambda\\
&\phantom{=}+(C_0-A_1)\lambda^2\\
M(\lambda)&=\left(\frac{1}{2}AA_1+\frac{1}{2}AC_0-BA_0+A_2-B_1\right)\\
&\phantom{=}+\left(\frac{1}{2}A_1-\frac{1}{2}C_0\right)\lambda,
\end{aligned}
$$
where $A_i$, $B_i$, $C_i$ and $D_i$ denote $\p_iA$, $\p_iB$, $\p_iC$ and $\p_iD$, respectively, and $\lambda$ is an auxiliary spectral coordinate.
\end{thm}
\begin{rmk}
The spectral parameter $\lambda$ can be treated as an affine parameter on the fibres of $\mathcal{C}$. The theorem states that $\mathcal{D}=\mathrm{span}\{L_0,L_1\}$ is an integrable rank-2 distribution on $\mathcal{C}$. There is a 3-parameter family of integral manifolds of $\mathcal{D}$. Projections of these submanifolds to $M$ give a 3-parameter family of 2-dimensional submanifolds of $M$ tangent to the field of cones $\tilde{\mathcal{C}}$. 
\end{rmk}
\begin{rmk}
A Cartan--K\"ahler analysis reveals that the first order system~\eqref{intSystem} -- or equivalently~\eqref{eq:mainpde} -- is involutive and has solutions depending on four functions of three variables, confirming the count of Bryant~\cite{MR1141197}. Moreover, easy computations show that the characteristic variety of system \eqref{intSystem} linearised along any solution $(A,B,C,D)$ is the discriminant locus, i.e. the tangential variety of $\tilde{\mathcal{C}}$.
\end{rmk}
\begin{proof}[Proof of Theorem~\ref{thm:system}]
The system \eqref{intSystem} can be directly obtained by expanding \eqref{eq:mainpde} explicitly in terms of the functions $A,B,C,D$. Here we use a different method and apply \cite[Corollary 7.4]{K0} to the framing $(V_0,3V_1,3V_2,V_3)$. Namely, denoting $\lambda=\frac{s}{t}$ we get that the curve $\mathcal{C}$ in $\mathbb{P}(TM)$ is the image of $\lambda\mapsto \R V(\lambda)\in \mathbb{P}(TM)$, where $V(\lambda)=\lambda^3 V_0+3\lambda^2 V_1+3\lambda V_2+V_3$ and the vector fields $V_0$, $V_1$, $V_2$ and $V_3$ are given by \eqref{eqVf} with
$$
\alpha=-A,\quad\beta=-B+A^2,\quad\gamma=-C+A^2,\quad\delta=-D+(C+B)A-A^3.
$$
According to \cite[Corollary 7.2]{K0}, a $\G$-structure is torsion-free if and only if
\begin{equation}\label{eq:lieBracket}
\left[V(\lambda),\frac{d}{d\lambda}V(\lambda)\right]\in\mathrm{span}\left\{V(\lambda),\frac{d}{d\lambda}V(\lambda), \frac{d^2}{d\lambda^2}V(\lambda)\right\},
\end{equation}
for any $\lambda\in\R$. This, due to \cite[Corollary 7.4]{K0} applied to the framing $(V_0,3V_1,3V_2,V_3)$, is expressed as eight linear equations for structural functions $c_{ij}^k$ defined by $[V_i,V_j]=\sum_kc_{ij}^kV_k$. However, in the present case the vector fields $V_i$ are special and four equations are void. Indeed, the nontrivial equations are as follows: 
$$
c^0_{23}=0, \qquad c^1_{23}-2c^0_{13}=0
$$ and
$$
c^2_{23}-2c^1_{13}+c^0_{03}+3c^0_{12}=0, \qquad c^3_{23}-2c^2_{13}+c^1_{03}+3c^1_{12}-2c^0_{02}=0
$$ (the equations differ from equations in \cite{K0} because of the factor 3 next to $V_1$ and $V_2$ in the present paper). Substituting the structural functions, which can be easily computed, we get the system \eqref{intSystem}.

Now, we consider
$$
L_0=V(\lambda)-\left(\lambda-\frac{1}{3}A\right)\frac{d}{d\lambda}V(\lambda)\mod \p_\lambda
$$
and
$$
L_1=\frac{1}{3}\frac{d}{d\lambda}V(\lambda)\mod \p_\lambda.
$$
Due to \eqref{eq:lieBracket}, the commutator $[L_0,L_1]$ lies in the span of $\{L_0,L_1,\frac{d^2}{d\lambda^2}V(\lambda)\}\mod\partial_\lambda$. Moreover, since 
$$
L_0=\p_3\mod\p_1,\p_0,\p_\lambda 
$$
and 
$$
L_1=\p_2\mod\p_1,\p_0,\p_\lambda
$$ 
we get $[L_0,L_1]=\varphi\frac{d^2}{d\lambda^2}V(\lambda)\mod\p_\lambda$ for some $\varphi$. One checks by direct computations that $N(\lambda)$ and $M(\lambda)$ are chosen such that $\varphi=0$ and the coefficient of $[L_0,L_1]$ next to $\partial_\lambda$ vanishes as well.
\end{proof}

\providecommand{\bysame}{\leavevmode\hbox to3em{\hrulefill}\thinspace}
\providecommand{\noopsort}[1]{}
\providecommand{\mr}[1]{\href{http://www.ams.org/mathscinet-getitem?mr=#1}{MR}}
\providecommand{\zbl}[1]{\href{http://www.zentralblatt-math.org/zmath/en/search/?q=an:#1}{Zbl~#1}}
\providecommand{\jfm}[1]{\href{http://www.emis.de/cgi-bin/JFM-item?#1}{JFM~#1}}
\providecommand{\arxiv}[1]{\href{http://www.arxiv.org/abs/#1}{arXiv~#1}}
\providecommand{\doi}[1]{\href{http://dx.doi.org/#1}{DOI}}
\providecommand{\MR}{\relax\ifhmode\unskip\space\fi MR }
\providecommand{\MRhref}[2]{%
  \href{http://www.ams.org/mathscinet-getitem?mr=#1}{#2}
}
\providecommand{\href}[2]{#2}

\end{document}